\documentclass[12pt,leqno,fleqn]{amsart}  
\usepackage{amsmath,amstext,amsthm,amssymb,amsxtra}

\usepackage[color=green!15,bordercolor=red,linecolor=red!30]{todonotes}
\usepackage[normalem]{ulem}

\usepackage{txfonts} 
\usepackage[T1]{fontenc}
\usepackage{lmodern}

\usepackage{euler}   

\usepackage{tikz}

\usepackage{mathtools}
\mathtoolsset{showonlyrefs,showmanualtags}

\usepackage{hyperref} 
\hypersetup{
    colorlinks=true,       
    linkcolor=blue,          
    citecolor=magenta,        
    filecolor=magenta,      
    urlcolor=cyan           
}

\usepackage[msc-links]{amsrefs}  

\allowdisplaybreaks
\swapnumbers

\theoremstyle{plain} 
\newtheorem{lemma}[equation]{Lemma} 
\newtheorem{proposition}[equation]{Proposition} 
\newtheorem{theorem}[equation]{Theorem}

\theoremstyle{definition}
\newtheorem{definition}[equation]{Definition} 

\theoremstyle{remark}
\newtheorem{remark}[equation]{Remark}

\newtheorem*{ack}{Acknowledgment}

\numberwithin{equation}{section}

	\title[Two Weight Hilbert Transform]{Two Weight Inequality for the Hilbert Transform:\\ A Real Variable Characterization, II}
\author[MT Lacey]{Michael T. Lacey}   

\address{ School of Mathematics, Georgia Institute of Technology, Atlanta GA 30332, USA}
\email {lacey@math.gatech.edu}
\thanks{Research supported in part by grant NSF-DMS 0968499, 
and  a grant from the Simons Foundation (\#229596 to Michael Lacey). 
The  author benefited from the research program Operator Related Function Theory and Time-Frequency Analysis at the Centre for Advanced Study at the Norwegian Academy of Science and Letters in Oslo during 2012---2013.}

\begin{document}

\begin{abstract}
Let $\sigma $ and $w $ be locally finite positive Borel measures on $\mathbb{R}$ which do not share a common point mass.  
Assume that the pair of weights satisfy a  Poisson $ A_2$ condition, and satisfy the testing conditions below, for the Hilbert 
transform $ H$, 
\begin{equation*}
\int _{I}   H (\sigma  \mathbf 1_{I}) ^2 \;d w \lesssim \sigma (I) ,  
\qquad 
\int _{I}   H (w  \mathbf 1_{I}) ^2\; d \sigma  \lesssim w (I) ,  
\end{equation*}
with constants independent of the choice of interval $ I$.  Then $ H (\sigma \,\cdot )$ maps $ L ^2 (\sigma )$ to $ L ^2 (w)$, verifying a conjecture of Nazarov--Treil--Volberg. 
The proof uses basic tools of non-homogeneous analysis with two  components particular to the Hilbert transform. The first is a global to local reduction, a consequence of prior work of Lacey-Sawyer-Shen-Uriarte-Tuero.  The second, an analysis of the local part, is the contribution of this paper.  
\end{abstract}

\maketitle 
\setcounter{tocdepth}{1}
\tableofcontents 

\section{Introduction} 

This paper continues \cite{12014319}, completing a real variable characterization of the two weight inequality for the Hilbert transform, 
formulated here.  
Given weights (i.e.\thinspace 
locally bounded positive Borel measures) $\sigma $ and $w $ on the real line $\mathbb{R}$, we consider the following \emph{two weight norm inequality for the Hilbert transform,}
\begin{equation}
\sup _{0 < \epsilon < \delta }
\int_{\mathbb{R}}\vert H _{\epsilon , \delta }( f\sigma ) \vert ^{2}\;w (dx)
\le \mathscr N ^2 \int_{\mathbb{R}}\vert f\vert ^{2}\; \sigma (dx) ,
\qquad  f\in L^{2}( \sigma ) ,  \label{2wtH}
\end{equation}%
where $\mathscr{N}$ is the best constant in the inequality,   uniform over all $ 0<\epsilon < \delta $,  
which define a standard truncation of the Hilbert transform applied to a signed  locally finite measure $ \nu $, 
\begin{equation*}
H _{\epsilon } \nu (x) := \int _{\epsilon < \lvert  x-y\rvert < \delta   } \frac {\nu (dy)} {y-x} \,. 
\end{equation*}
We insist upon this formulation as the principal value need not exist in the generality that we are interested in.  
Below, however, we systematically suppress the uniformity over $ \epsilon, \delta  $ above, writing just $ H $ for $ H _{\epsilon, \delta }$, 
understanding  that all estimates are independent of $ 0< \epsilon <  \delta $.

A question of fundamental importance is establishing characterizations of the inequality above. 
In this paper we complete the proof of a conjecture of Nazarov-Treil-Volberg \cites{10031596,V}. 
Set 
\begin{equation} \label{e:Poisson}
P (\sigma ,I) := \int _{\mathbb R } \frac {\lvert  I\rvert } { \lvert  I\rvert ^2 + \textup{dist} (x,I) ^2  } \; \sigma (dx), 
\end{equation}
which is, essentially, the usual Poisson extension of $ \sigma $ to the upper half plane, evaluated at $ (x_I, \lvert  I\rvert) $, 
where $ x_I$ is the center of $ I$.

\begin{theorem}\label{t:H}  Let $\sigma $ and $w $ be locally finite
positive Borel measures on the real line $\mathbb{R}$ with no common point
masses. Then, the two weight inequality \eqref{2wtH} holds if and only if these 
three conditions hold uniformly over all intervals $ I$, 
\begin{gather} \label{e.A2}
	P( \sigma, I )P( w, I ) \leq \mathscr A_2, 
	\\  \label{e:testing}
\int_{I}\vert H( \mathbf{1}_{I}\sigma ) \vert
^{2}\; d w \leq\mathscr{T} ^2 \sigma ( I), \qquad 
\int_{I}\vert H( \mathbf{1}_{I}w ) \vert
^{2}\; d\sigma \leq \mathscr{T} ^2 w ( I). 
\end{gather} 
There holds 
\begin{equation} \label{e:Hdef}
	\mathscr{N}\approx \mathscr A_{2} ^{1/2}+\mathscr{T} =: \mathscr H ,  
\end{equation}
where $ \mathscr A_2$ and $ \mathscr T$ are the best constants in the inequalities above. 
\end{theorem}

The first condition is an extension of the typical $ A_2$ condition to a Poisson setting, which is known to be necessary. 
The second condition \eqref{e:testing} is called an `interval testing condition', and is obviously necessary.        
Thus, the content of the Theorem is the sufficiency of the $ A_2$ and testing conditions for the norm inequality. 
We refer the reader to the introduction of \cite{12014319} for a history of the problem and indications of how the 
question arises in the setting of analytic function spaces, operator theory, and spectral theory.  

In Part 1, \cite{12014319}, the proof of the sufficiency was reduced to a `local' estimate.  
Herein, we complete the proof of the local estimate.  Relevant notations and conventions are contained in Part 1.

\begin{ack}
This paper has been improved by the generous efforts of the referees.  
\end{ack}

\section{The  Local Estimate} \label{s:global}

We recall the local estimate.  Throughout, $ \mathscr H := \mathscr A_2 ^{1/2} + \mathscr T$, 
and all intervals are in a fixed dyadic grid $ \mathcal D$, for which neither $ \sigma $ nor $ w$ have a point mass at an end point of $ I$. 

\begin{definition}\label{d:energy} Given any  interval $ I_0$, define $ \mathcal F_{\textup{energy}} (I_0)$ to be  
	the maximal subintervals $ I \subsetneq I_0$ such that     
\begin{equation}\label{e.Estop}
 P(\sigma I_0, I) ^2 \mathsf E (w,I) ^2 w (I) > 10 C_0{\mathscr H} ^2 \sigma (I) \,. 
\end{equation}
There holds $ \sigma (\cup \{F \,:\, F\in \mathcal F (I_0)\}) \le \tfrac 1 {10} \sigma (I_0)$.  
\end{definition}

\begin{definition}\label{d:BF} Let $ I_0$ be an interval, and let $ \mathcal S$ be a collection of disjoint intervals contained in $ S$. 
A function $ f \in L ^2 _0 (I_0, \sigma )$ is said to be  \emph{uniform (w.r.t.\thinspace  $ \mathcal S$)} if 
these conditions are met: 
\begin{enumerate}
\item  Each energy stopping interval $ F\in \mathcal F _{\textup{energy}} (I_0)$ is contained in some $ S\in \mathcal S$. 
\item The function $ f$ is constant on each interval  $ S \in \mathcal S$. 
\item  For any interval $ I$ which is not contained in any $ S\in \mathcal S$, 
$  \mathbb E ^{\sigma }_I \lvert f \rvert \le 1 $.  
\end{enumerate}
We will say that $ g$ is \emph{weakly adapted} to a function $ f$ uniform w.r.t.\thinspace $ \mathcal S$, if 
for all intervals $ J$ with $ \langle g, h ^{w}_J\rangle _{w} \neq 0$, we have $ J\not\Subset S$ for all  $ S\in \mathcal S$.
We will also say that $ g$ is \emph{weakly adapted to $ \mathcal S$.} 
\end{definition}

Define the bilinear form  
\begin{equation*}
B ^{\textup{above}} (f,g) := \sum_{I \;:\; I\subset I_0} \sum_{J \;:\; J\Subset I} 
\mathbb E ^{\sigma } _{J} \Delta ^{\sigma }_I f \cdot   \langle H _{\sigma } I_J, \Delta ^{w} _{J} g \rangle _w 
\end{equation*}
In the sum above, both $ I$ and $ J$ can be further restricted to be good. Goodness of both is important below. 
  The constant $ \mathscr L$ is defined as the best constant in the \emph{local estimate}, as written below, or in its dual form 
  with the roles of $ \sigma $ and $ w$ interchanged. 
\begin{equation}  \label{e:BF}
\lvert  B ^{\textup{above}} (f,g)\rvert \le \mathscr L \{\sigma (I_0) ^{1/2} +  \lVert f\rVert_{\sigma }\} \lVert g\rVert_{w} , 
\end{equation}
where $ f, g$  of mean zero on their respective spaces, supported on an interval $ I_0$. 
Moreover,  $ f$ is  uniform.  and $ g$ is weakly adapted to $ f$.  
The inequality above is homogeneous in $ g$, but not $ f$, since the term $ \sigma (I_0) ^{1/2} $ is motivated by the bounded averages property of $ f$.    

The main result of \cite{12014319} is this provisional estimate on the norm of the two weight Hilbert transform: $ \mathscr N \lesssim \mathscr H + \mathscr L$.  Herein, we complete the proof of the Nazarov-Treil-Volberg conjecture by showing that 

\begin{theorem}
\label{t:local} There holds $ \mathscr L \lesssim \mathscr H$. 
\end{theorem}

Let $ f$ be adapted to $ \mathcal S$ on interval $ I_0$. 
The bounded averages property in the definition of uniformity is used to make the following  routine appeal to the testing condition.
Focusing on the argument of the Hilbert transform in \eqref{e:BF}, we write $ I_J = I_0 - (I_0 - I_J)$. 
When the interval is $ I_0$, and $ J$ is in the Haar support of $ g$, notice that the scalar 
\begin{equation*}
\varepsilon _J := \sum_{I \;:\; J\Subset I_J \subset I_0} 
\mathbb E ^{\sigma } _{J} \Delta ^{\sigma } _{I} f   
\end{equation*}
is bounded by one, as we now argue. 
Say that $ f$ is uniform w.r.t.\thinspace $ \mathcal S$, and  let $ I ^{-}$ be the minimal interval in the Haar support of $ f$ with $ J\Subset I$.  
Since $ g$ is weakly adapted to $ f$, we cannot have $ I_J ^{-} $ contained in an interval  $ S \in \mathcal S$, and so $\lvert   \mathbb E ^{\sigma } _{I_J ^{-}} f \rvert \le 1 $. 
By the telescoping identity for martingale differences, 
\begin{equation*}
\varepsilon _J = \sum_{I \;:\;  I ^{-}\subset I \subset I_0} 
\mathbb E ^{\sigma } _{I_J} \Delta ^{\sigma } _{I} f =  \mathbb E ^{\sigma } _{I_J ^{-}} f ,    
\end{equation*}
which is at most one in absolute value.  

Therefore, we can write 
\begin{align*}
\Bigl\lvert \sum_{I \;:\; I\subset I_0} \sum_{J \;:\; J\Subset I} 
 \mathbb E ^{\sigma } _{J} \Delta ^{\sigma } _{I} f \cdot \langle H _{\sigma } I_0, \Delta ^{w} _{J} g\rangle  \Bigr\rvert
 & = 
\Bigl\lvert 
\bigl\langle  H _{\sigma } I_0 ,  \sum_{\substack{J \;:\; J\Subset I_0 }}  \varepsilon _J \Delta ^{w} _{J} g \bigr\rangle_{w}
\Bigr\rvert
 \\
 & \le \mathscr T \sigma (I_0) ^{1/2} 
 \Bigl\lVert    \sum_{\substack{J \;:\; J\Subset I_0 }}  \varepsilon _J \Delta ^{w} _{J} g  \Bigr\rVert_{w}
 \\
 & \le  \mathscr T \sigma (I_0) ^{1/2}  \lVert g\rVert_{w} \,. 
\end{align*}
This uses only interval testing  and orthogonality of the martingale differences, and it matches the first half of the right hand side of \eqref{e:BF}.  

\smallskip 

This leaves the case of the argument of the Hilbert transform being $ I_0 - I_J$.  When the argument of the Hilbert transform is $ I_0 - I_J$, this is the 
\emph{stopping form}, the last component  of the local part of the problem.   

Switch focus to the function $ g$. Recall that $ g$ is weakly adapted, in the sense that 
for each interval $ J$ with $ \langle g, h^{w} _{J}  \rangle_w \neq 0$ implies that $ J$ is \emph{not} strongly contained in 
an interval $ F\in \mathcal F _{\textup{energy}} (I_0)$.   We address here this subcase: 
Assume that $ \langle g, h^{w} _{J}  \rangle_w \neq 0$ implies that  $ J$ is necessarily contained in some $ F\in \mathcal F _{\textup{energy}} (I_0)$.    

Hold an integer  $0\le s <r$ fixed, and let $ \mathcal J_s $ be the intervals $ J$ 
in the Haar support of $ g $ so that for some $ F\in \mathcal F _{\textup{energy}} (I_0)$, 
$ J\subset F$ and $ 2 ^{s} \lvert  J\rvert= \lvert  F\rvert  $.  
The union of the collections $ \mathcal J _{s}$, for $ 0\le s < r$ exhaust the Haar support of $ g$. 
But, we can then estimate uniformly in $ 0\le s < r$, 
\begin{align*}
\Bigl\lvert  \sum_{J\in \mathcal J_s}  
 \sum_{I \::\: J\Subset I_J \subset I_0} 
\mathbb E ^{\sigma } _{I_J} \Delta ^{\sigma } _{I} f \cdot \langle  H _{\sigma } (I_0-I_J),  \Delta ^{\sigma } _{J} g  \rangle_w 
\Bigr\rvert 
&\lesssim  
  \sum_{J\in \mathcal J_s}  P (\sigma \cdot I_0 , J) E (w,J) w (J) ^{1/2} \lvert \langle g, h ^{w} _{J} \rangle_w  \rvert 
\\   
& \lesssim \sigma (I_0) ^{1/2} \lVert g\rVert_{w}.  
\end{align*}
The bounded averages property of $ f$, and the   the monotonicity principle 
permit the domination by the Poisson terms above. Then, Cauchy--Schwarz and the energy inequality are applied. 

\smallskip 
We are then left with the case that $ g$ is constant on each interval $ F\in \mathcal F _{\textup{energy}} (I_0)$.  
This is the delicate case  that is taken up in the next section.

\section{The Stopping Form} \label{s:stop}

Given an interval $ I_0$, the stopping form is 
\begin{equation}  \label{e:stop}
B ^{\textup{stop}} _{I_0} (f,g) := 
 \sum_{I \;:\; I\subset I_0} \sum_{J \;:\; J\Subset I_J} 
 \mathbb E ^{\sigma } _{I_J} \Delta ^{\sigma }_I f \cdot \langle H _{\sigma } (I_0 - I_J) ,  \Delta ^{w} _{J} g\rangle _{w}\,. 
\end{equation}	
We prove the estimate below for the stopping form, which completes the proof of Theorem~\ref{t:local}.   
Note that the hypotheses on $ f$ and $ g$ are that they are adapted to energy stopping intervals. (Bounded averages on $ f$ are no longer required.) 

\begin{lemma}\label{l:stop<} Fix an interval $ I_0$, and  suppose that $ f$ and 
$ g  $ are constant on each interval $ F\in \mathcal F _{\textup{energy}} (I_0)$. 
Then, 
\begin{equation}\label{e:stop<}
\lvert  B ^{\textup{stop}} _{I_0} (f,g) \rvert \lesssim \mathscr H \lVert f\rVert_{\sigma } \lVert g\rVert_{w} \,.  
\end{equation}
\end{lemma}

The stopping form arises naturally in any proof of a $ T1$ theorem using Haar or other bases.  
In the non-homogeneous case, or in the $ Tb$ setting, where (adapted) Haar functions are important tools, it frequently appears in more or less this form.  
Regardless of how it arises, the stopping form is treated as a error, in that it is bounded by some simple geometric series, obtaining decay as e.\thinspace g.\thinspace the ratio $ \lvert  J\rvert/\lvert  I\rvert  $ is held fixed. 
(See for instance \cite{10031596}*{(7.16)}.) 

These sorts of arguments, however, implicitly require some additional hypotheses, such as the  weights being  mutually $ A _{\infty }$. 
Of course, the two weights above can be mutually singular. There is no \emph{a priori}  control of the stopping form in terms of simple parameters  like $ \lvert  J\rvert/\lvert  I\rvert  $, even supplemented by  additional pigeonholing of various parameters.

Our method is inspired by proofs of Carleson's Theorem on Fourier series \cites{lacey-thiele-carleson,fefferman,MR0199631}.

\subsection{Admissible Pairs}

A range of decompositions  of the stopping form necessitate a somewhat heavy notation that we introduce here. 
The individual summands in the stopping form involve four distinct intervals, namely $ I_0, I, I_J$, and $ J$.  
The interval $ I_0$ will not change in this argument, and the pair $ (I,J)$ determine $ I_J$. 
Subsequent decompositions are  easiest to phrase as actions on collections $ \mathcal Q$ of pairs of intervals 
$
Q= (Q_1, Q_2)  $ with $Q_1\Supset Q_2  
$. 
(The letter $ P $ is already taken for the Poisson integral.) 
And we consider the bilinear forms 
\begin{equation*}
B _{\mathcal Q} (f,g) := \sum_{Q\in \mathcal Q}    \mathbb E ^{\sigma } _{Q_2} \Delta ^{\sigma }_ {Q_1} f \cdot \langle H _{\sigma } (I_0 - (Q_1)_{Q_2}) ,  \Delta ^{w} _{Q_2} g\rangle _{w} \,. 
\end{equation*}

We will have the standing assumption that   all collections $ \mathcal Q$ that we consider are \emph{admissible}.  

\begin{definition}\label{d:admiss} A collection of pairs $ \mathcal Q$ is \emph{admissible} if it meets these criteria. 
For any $ Q = (Q_1, Q_2) \in \mathcal Q$, 
\begin{enumerate}
\item    $Q_2\Subset (Q_1) _{Q_2}\subset I_0$, and $ Q_1, Q_2$ are good. 
 \item (convexity in $ Q_1$) If  $ Q''\in \mathcal Q$ with $ Q''_2=Q_2$ and  $ Q_1'' \subset I\subset Q_1$,  and $ I$ is good, then there is a $ Q' \in \mathcal Q$ with $ Q '_1= I$ and $ Q_2'=Q_2$.  
\end{enumerate}
The first  property is self-explanatory. The second property is convexity in $ Q_1$, holding $ Q_2$ fixed, which is used in the estimates 
on the stopping form which conclude the argument.  Keep in mind that $ f$ is assumed to be good, meaning that its Haar support only contains 
good intervals, thus convexity is the natural condition.  
A third property is described below. 

We exclusively use the notation $ \mathcal Q _{k}$, $ k=1,2$ for the collection of intervals $ \bigcup \{ Q_k \;:\; Q\in \mathcal Q\}$, not  counting multiplicity.  Similarly, set $ \tilde {\mathcal Q}_1 := \{ (Q _{1}) _{Q_2} \;:\; Q\in \mathcal Q\}$, and $ \tilde Q_1 := (Q_1)_{Q_2}$.  
\begin{enumerate}
\item[(3)] No interval $ K\in \mathcal Q_2$  is  contained in an interval $ S\in \mathcal F _{\textup{energy}} (I_0)$.  
(And so,  no interval $ K \in   \tilde {\mathcal Q}_1 $ is contained in an interval  $ S\in \mathcal F _{\textup{energy}} (I_0)$.)
\end{enumerate}

\end{definition}

The last requirement comes from the assumption that the  functions $ f$ and $ g$ be constant on the intervals in  $ \mathcal F _{\textup{energy}} (I_0)$. 
We will be appealing to different Hilbertian arguments below, so we prefer to make this  an assumption about the pairs than the functions $ f, g$. 
(The Hilbert space will be that of good functions in $ L ^2 (\sigma )$ and $ L ^2 (w)$.) 
Both $ f$ and $ g$ are good, and in particular, goodness of $ f$ is exploited below. 
Goodness permits  estimates of off-diagonal inner products involving the Hilbert transform by  Poisson averages, 
and regularizes the Poisson averages, see \S\ref{s:upper}.

The stopping form is obtained with the admissible collection of pairs given by 
\begin{equation} \label{e:Q0}
\mathcal Q_0 =\{ (I, J) \;:\;   J\Subset I_J,  \textup{$ I $ and $J$ are good,}\  J \not\Subset  S \textup{ $ $ for all } S\in \mathcal S\} \,. 
\end{equation}
In this definition $ \mathcal S$ is the collection of subintervals of $ I_0$ which $ f$ is uniform with respect to.  
There holds $ B ^{\textup{stop}} _{I_0} (f,g) = B _{\mathcal Q_0} (f,g)$ for $ f ,g$ constant on the intervals in  $ \mathcal F _{\textup{energy}} (I_0)$. 

\medskip   

There is a very important notion of the size of $ \mathcal Q$.  
\begin{equation}\label{e:size}
\textup{size} (\mathcal Q) ^2  := 
\sup _{ K \in \tilde {\mathcal Q}_1 \cup \mathcal Q_2}   \frac {P(\sigma  (I_0 -K), K) ^2 } {\sigma (K) \lvert  K\rvert ^2  } 
\sum_{J \in \mathcal Q _2 \;:\; J\subset K}  \langle  x , h ^{w} _{J} \rangle_{w} ^2 \,. 
\end{equation}
We only form the supremum over intervals $ \mathcal T_ {\mathcal P}$, which are not contained in an energy stopping interval. 
For admissible $ \mathcal Q$, there holds $ \textup{size} (\mathcal Q) \lesssim \mathscr H$, as follows the  property (3) in Definition~\ref{d:admiss}, and Definition~\ref{d:energy}.

More definitions follow.  
Set the norm of the bilinear form $ \mathcal Q$ to be the best constant in the inequality 
\begin{equation*}
\lvert  B _{\mathcal Q} (f,g)\rvert \le \mathbf B _{\mathcal Q} \lVert f\rVert_{\sigma } \lVert g\rVert_{w } \,.  
\end{equation*}
Thus, our goal is show that $ \mathbf B _{\mathcal Q} \lesssim \textup{size} (\mathcal Q)$ for admissible $ \mathcal Q$, but we will only be able to do this directly in the case that the pairs $ (Q_1, Q_2)$ are weakly decoupled.

Say that collections of pairs $ \mathcal Q ^{j}$, for $ j\in \mathbb N $, are \emph{mutually orthogonal} if  on the one hand, the collections  $ (\mathcal Q ^{j}) _{2}$ are pairwise disjoint,  and on the other, that the collection $ \widetilde {(\mathcal Q ^{j})}_1$ are pairwise disjoint. 
The concept has to be different in the first and second coordinates of the pairs, due to the different role of the intervals $ Q_1$ and $ Q_2$. 
The reader should note that a given interval $ I$ can be in two, but not more, distinct collections $ \mathcal Q ^{j}_1$, since 
mutual orthogonality is determined by the two children of $ I$.  

The meaning of  mutual orthogonality is best expressed through the norm of the associated bilinear forms. 
Under the assumption that $ B _{\mathcal Q} = \sum_{j\in \mathbb N } B _{ \mathcal Q ^{j}}$, 
and that the $ \{\mathcal Q ^{j} \;:\; j\in \mathbb N \} $ are mutually orthogonal, the following essential inequality holds. 
\begin{equation}\label{e:subadd}
\mathbf B _{\mathcal Q} \le  \sqrt 2\sup _{j\in \mathbb N } \mathbf B _{\mathcal Q ^{j}} \,. 
\end{equation}
Indeed, for $ j\in \mathbb N $, let $ \Pi ^{w} _{j}$ be the projection onto the linear span of the Haar functions $\{ h ^{w} _{J} \;:\; J\in \mathcal Q ^{j}_2\}$, and   $ \Pi ^{\sigma  }_j $ is the projection onto the span of  $\{ h ^{\sigma } _{I} \;:\; I\in \mathcal Q ^{j}_1\}$.  We then have the two inequalities 
\begin{equation*}
\sum_{j \in \mathbb N } \lVert \Pi ^{w}_j g \rVert_{w} ^2 \le \lVert g\rVert_{w} ^2 , \qquad 
\sum_{j\in \mathbb N } \lVert \Pi ^{\sigma }_j f \rVert_{\sigma } ^2 \le  2\lVert f\rVert_{\sigma} ^2 \,. 
 \end{equation*}
The first inequality is clear from the mutual orthogonality of the projections $ \Pi ^{w} _{j}$.  But, the projections $ \Pi ^{\sigma } _{j}$ 
are not orthogonal, but a given Haar function $ h ^{\sigma } _{I}$ is the range of at most two of them.  
Therefore, we have 
\begin{align*}
\lvert  B _{\mathcal Q} (f,g)\rvert  & \le \sum_{j \in \mathbb N }  \lvert  B _{\mathcal Q ^{j}} (f,g) \rvert 
\\
&=  \sum_{j\in \mathbb N }\lvert   B _{\mathcal Q ^{j}} (\Pi ^{\sigma }_jf, \Pi ^{w} _{j}g) \rvert 
\\
&\le \sum_{j\in \mathbb N } \mathbf B _{\mathcal Q ^{j}}  \lVert \Pi ^{\sigma }_j f \rVert_{\sigma }
 \lVert \Pi ^{w}_j g \rVert_{w} 
 \le   \sqrt 2 \sup _{j\in \mathbb N } \mathbf B _{\mathcal Q ^{j}}   \cdot \lVert f\rVert_{\sigma   } \lVert g\rVert_{w} \,. 
\end{align*}
This proves \eqref{e:subadd}.

\subsection{The Recursive Argument}

This is the essence of the matter.

\begin{lemma}\label{l:Decompose}[Size Lemma] 
An admissible  collection of pairs $ \mathcal Q$  can be partitioned into  collections 
$ \mathcal Q ^{\textup{large}}$ and admissible  $ \mathcal Q ^{\textup{small}} _{t}$, 
for $  t \in \mathbb N $ such that 
\begin{gather}\label{e:BQ<}
\mathbf B _{\mathcal Q} \le 
C\textup{size} (\mathcal Q) + (1+ \sqrt 2) 
\sup _t
\mathbf B _{ \mathcal Q ^{\textup{small}} _{t}} ,
\\ \label{e:small2}
\text{and} \quad \sup _{t\in \mathbb N }\textup{size} ( \mathcal Q ^{\textup{small}} _{t}) \le \tfrac 14   \textup{size} (\mathcal Q) \,.
\end{gather}
Here, $ C>0 $ is an absolute constant. 
\end{lemma}

The point of the lemma is that all of the constituent parts are better in some way, and that the right hand side of \eqref{e:BQ<} involves  a favorable  supremum. We can quickly prove the main result of this section. 

\begin{proof}[Proof of Lemma~\ref{l:stop<}] 
The stopping form of this Lemma is of the form $ B _{\mathcal Q} (f,g)$ for admissible choice of $ \mathcal Q$, 
with $ \textup{size} (\mathcal Q) \le C\mathscr H$, as we have noted in \eqref{e:Q0}.  
Define 
\begin{equation*}
\zeta (\lambda ) := \sup \{  \mathbf B _{\mathcal Q} \;:\; \textup{size} (\mathcal Q) \le C \lambda \mathscr H \}, \qquad 0< \lambda \le 1,
\end{equation*}
where $ C>0$ is a sufficiently large, but absolute constant, and the supremum is over admissible choices of $ \mathcal Q$.  
We are free to assume that $ \mathcal Q_1$ and $ \mathcal Q_2$ are further constrained to be in some fixed, but large, collection of intervals $ \mathcal I$. 
Then, it is clear that $ \zeta (\lambda )$ is finite, for all $ 0< \lambda \le 1$.
Because of the way the constant $ \mathscr H$ enters into the definition, it remains to show that $ \zeta (1)$ admits an absolute upper bound, independent of how $ \mathcal I$ is chosen. 

\smallskip 

It is the consequence of Lemma~\ref{l:Decompose} that there holds 
\begin{align}
\zeta (\lambda) &\le C \lambda  +  (1+ \sqrt 2)  \zeta ( \lambda /4 ), \qquad 0 < \lambda \le  1 \,. 
\end{align}
Iterating this inequality beginning at $ \lambda  =1$ gives us 
\begin{align*}
\zeta (1 )& \le C +   (1+\sqrt 2)\zeta (1/4) \le \cdots \le C \sum_{t=0} ^{\infty } \bigl[\tfrac {1+ \sqrt 2} 4\bigr]  ^{t} \le  4C  \,. 
\end{align*}
So we have established an absolute upper  bound on $ \zeta (1)$.
\end{proof}

\subsection{Proof of Lemma~\ref{l:Decompose}}

  We restate the conclusion of Lemma~\ref{l:Decompose} to more closely follow the  line of argument to follow. 
The collection $ \mathcal Q$ can be partitioned into two collections $ \mathcal Q ^{\textup{large}}$ and $ \mathcal Q ^{\textup{small}}$ 
such that 
\begin{enumerate}
\item  $ \mathbf  B _{\mathcal Q ^{\textup{large}}}   \lesssim   \tau    $, where $ \tau = \textup{size} (\mathcal Q) $. 

\item $ \mathcal Q ^{\textup{small}} =\mathcal Q ^{\textup{small}}_1  \cup \mathcal Q ^{\textup{small}}_{2} $.  

\item The collection $\mathcal Q ^{\textup{small}}_1 $ is admissible, 
 and $ \textup{size} (\mathcal Q ^{\textup{small}} _1) 
\le \frac \tau 4  $.

\item For a collection of  dyadic intervals $ \mathcal L$,   the collection $ \mathcal Q ^{\textup{small}}_2$ is the union of mutually orthogonal admissible  collections $  \mathcal Q ^{\textup{small}}_{2,L}$, for $ L\in \mathcal L $,  with 
\begin{equation*}
 \textup{size} (\mathcal Q ^{\textup{small}} _{2,L} )  
\le  \tfrac \tau 4  , \qquad L\in \mathcal L\,. 
\end{equation*}

\end{enumerate}
Thus, we have by inequality \eqref{e:subadd} for mutually orthogonal collections, 
\begin{align*} 
\mathbf B _{\mathcal Q} & \le \mathbf B _{\mathcal Q ^{\textup{large}}}
+ \mathbf B _{\mathcal Q ^{\textup{small}}_1 \cup  \mathcal Q ^{\textup{small}}_2 }
\\ & \le 
 \mathbf B _{\mathcal Q ^{\textup{large}}} 
+ \mathbf B _{\mathcal Q ^{\textup{small}}_1 } + \mathbf B _{ \mathcal Q ^{\textup{small}}_2 }
\\& \le 
C \tau + (1+ \sqrt 2) \max\bigl\{\mathbf B _{\mathcal Q ^{\textup{small}}_1}, 
\sup _{ L\in \mathcal L } 
\mathbf B _{\mathcal Q ^{\textup{small}}_{2, L}}  \bigr\}  \,. 
\end{align*}
This, with the properties of size listed above prove Lemma~\ref{l:Decompose} as stated, after a trivial re-indexing. 

\bigskip

All else flows from this construction of  a  subset  $ \mathcal L $ of dyadic subintervals of $ I_0$.  
The initial intervals in $ \mathcal L$ are  the minimal  intervals $ K  \in \tilde {\mathcal Q}_1 \cup \mathcal Q_2$  such that 
\begin{equation} \label{e:Kdef}
\frac {P(\sigma  (I_0- K) , K) ^2   } {\lvert  K\rvert ^2  }
\sum_{J\in \mathcal Q_2 \;:\; J\subset K}  
\langle  x   , h ^{w} _{J} \rangle _{w} ^2
\ge \frac {\tau ^2 } {16}  \sigma (K)\,.  
\end{equation}
Since $ \textup{size} (\mathcal Q)=\tau $, there are such intervals $ K$. 

Initialize $ \mathcal S $ (for `stock' or `supply') to be all the dyadic intervals in $   \tilde {\mathcal Q}_1 \cup \mathcal Q_2 $ which  strictly contain at least one element of $ \mathcal L$. 
In the recursive step,   let $\mathcal L'$ be the minimal elements $ S\in \mathcal S$ such that 
\begin{equation}\label{e:Lconstruct}
\sum_{ J\in \mathcal Q_2 \;:\;  J \subset S }   
\langle   x , h ^{w } _{J} \rangle_{w} ^2 
\ge  \rho    \sum_{\substack{L\in \mathcal L \;:\; L\subset S\\ \textup{$ L$ is maximal} }}\sum_{ J\in \mathcal Q_2 \;:\;  J \subset L }   
\langle   x , h ^{w } _{J} \rangle_{w} ^2 , \qquad  \rho =  \tfrac  {17} {16} \,. 
\end{equation}
(The inequality would be trivial if $ \rho =1$.)  
If $ \mathcal L'$ is empty the recursion stops. 
Otherwise,  update $ \mathcal L \leftarrow \mathcal L \cup \mathcal L'$, 
and 
$
\mathcal S \leftarrow \{ K\in \mathcal S \;:\;  K \not\subset L\ \forall L\in \mathcal L\} 
$.

Once the recursion stops, report the collection $ \mathcal L$.   It has this crucial  property: 
For $ L\in \mathcal L$, and integers $ t\ge 1$, 
\begin{equation} \label{e:ddecay}
\sum_{ L' \;:\; \pi _{\mathcal L} ^{t} L'=L} \sum_{J \in \mathcal Q_2 \;:\; J\subset L'} 
\langle   x , h ^{w } _{J} \rangle_{w} ^2  
\le \rho ^{-t} 
 \sum_{J \in \mathcal Q_2 \;:\; J\subset L} 
\langle   x , h ^{w } _{J} \rangle_{w} ^2 \,. 
\end{equation}
Indeed, in the case of $ t=1$, this is the selection criterion for membership in $ \mathcal L$, and a simple induction proves the statement for all $ t\ge 1$.

\begin{remark}\label{r:L} The selection of $ \mathcal L$ can be understood as a familiar argument concerning Carleson measures, although there is no such object in this argument.  
Consider the measure $ \mu $ on $ \mathbb R ^2 _+$ given as a sum of point masses given by 
\begin{equation*}
\mu := \sum_{J \in \mathcal Q_2\;:\; J\subset I_0} \langle x, h ^{w} _{J} \rangle_w ^2 \delta _{ (x_J, \lvert  J\rvert )}, \qquad \textup{$ x_J$ is the center of $ J$.} 
\end{equation*}
The tent over $ L$ is the triangular region $ T_L := \{  (x,y) \;:\;  \lvert  x-x_L\rvert \le \lvert  L\rvert - y  \}$, so that 
\begin{equation*}
\mu (T_L) =  \sum_{J \in \mathcal Q_2 \;:\; J\subset L} 
\langle   x , h ^{w } _{J} \rangle_{w} ^2 \,. 
\end{equation*}
Then, the selection rule for membership in $ \mathcal L $ can be understood as taking the minimal tent $ T _{L}$ such that 
$ \mu (T_L)$ is bigger than $ \rho $ times the $ \mu $-measure of the selected tents.  See Figure~\ref{f:tents}.  
\end{remark}

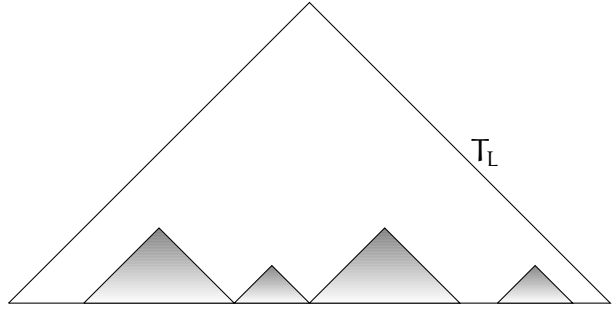
\begin{figure}
 \begin{tikzpicture}
 \draw (0,0) -- (4,4) -- (8,0)  node[right,midway] () {$ T_L$} -- (0,0); 
 \shadedraw (1,0) -- (2,1) -- (3,0) -- (1,0); 
  \shadedraw (3,0) -- (3.5,.5) -- (4,0) -- (3,0); 
   \shadedraw (4,0) -- (5,1) -- (6,0) -- (4,0);  
   \shadedraw (6.5,0) -- (7,.5) -- (7.5,0) -- (6.5,0);  
 \end{tikzpicture}
\caption{The shaded smaller tents have been selected, and $ T_L$ is the minimal tent with $ \mu (T_L)$ larger than $ \rho $ times 
the $ \mu $-measure of the shaded tents.}
\label{f:tents}
\end{figure}

The decomposition of $ \mathcal Q$ is based upon the relation of the pairs to the collection $ \mathcal L$, namely a pair $ \tilde Q_1 ,Q_2$ can (a) both have the same parent in $ \mathcal L$; (b) have distinct parents in $ \mathcal L$; (c) $ Q_2$ can have a parent in $ \mathcal L$, but  not $ \tilde Q_1$; and  (d) $ Q_2$ does not have a parent in $ \mathcal L$.  

A particularly vexing aspect of the stopping form is the linkage between the martingale difference on $ g$,  which is given by $ J$,  and the argument of the Hilbert transform, $ I_0- I_J$.  
The  `large' collections constructed below will, in a certain way,  decouple the   $ J$ and the $ I_0-I_J$, enough so that norm of the associated bilinear form can be estimated by the size of $ \mathcal Q$. 

In the `small' collections, there is however no decoupling, but critically,  the size of the collections is smaller, and by \eqref{e:BQ<}, we need 
only estimate the largest operator norm among the small collections.

\subsubsection*{Pairs comparable to $ \mathcal L$}
Define 
\begin{equation*}
\mathcal Q   _{L,t} := 
\{  Q\in \mathcal Q \;:\;  \pi _{\mathcal L} \tilde Q_1 = \pi ^{t} _{\mathcal L} Q_2=L \} , \qquad L\in \mathcal L,\ t\in \mathbb N \,.  
\end{equation*}
These are admissible collections, as the convexity property in $ Q_1$, holding $ Q_2$ constant, is clearly inherited from $ \mathcal Q$. 
Now, observe that for each $ t\in \mathbb N $, the collections $ \{ \mathcal Q  _{L,t} \;:\; L\in \mathcal L\}$ 
are mutually orthogonal:  
The collection of intervals $ ( \mathcal Q  _{L,t}) _{2}$ are obviously disjoint in $ L\in \mathcal L$, with $ t\in \mathbb N $ held fixed.  
And, since membership in these collections is determined in the first coordinate by the interval $ \tilde  Q_1$, and the two children of $ Q_1$ can have two different parents in $ \mathcal L$, a given interval $ I$ can appear in at most two collections  $ ({\mathcal Q  _{L,t}}) _{1}$, 
as $ L\in \mathcal L$ varies, and $ t \in \mathbb N $  held fixed.

Define $  \mathcal Q ^{\textup{small}}_1 $ to be  the union over $ L\in \mathcal L$ of the collections 
\begin{equation*}
  \mathcal Q ^{\textup{small}} _{L,1} := \{ Q \in  \mathcal Q _ { L , 1}    \;:\; \tilde Q_1 \neq L\} \,. 
\end{equation*}
Note in particular that we have only  allowed $ t=1$ above, and $ \tilde Q_1 =L$ is not allowed. 
For these collections, we need only verify that 
\begin{lemma}\label{l:small1}  There holds 
\begin{equation} \label{e:small1<}
\textup{size}  (\mathcal Q  ^{\textup{small}}_ { L,1} ) \le \sqrt { (\rho -1)}  \cdot \tau =\frac \tau {4}  , \qquad L\in \mathcal L ,\ t \in \mathbb N \,.  
\end{equation}
\end{lemma}
\begin{proof}
An interval $ K\in \widetilde {(\mathcal Q  ^{\textup{small}}_ { L,1}) }_1\cup \mathcal Q_2 $ is not in $ \mathcal L$, by construction. 
Suppose that $ K$ does not contain any interval in $ \mathcal L$.  By the selection of the initial intervals in $ \mathcal L$, 
the minimal intervals in $ \tilde {\mathcal Q}_1 \cup \mathcal Q_2$ which satisfy \eqref{e:Kdef}, 
it follows that the interval $ K$ must fail  \eqref{e:Kdef}.  And so we are done.  

Thus, $ K$ contains some element of $ \mathcal L$, whence the inequality \eqref{e:Lconstruct} must fail. 
Namely, rearranging that inequality, 
\begin{equation} \label{e:18}
 \sum_{\substack{ J\in \mathcal Q_2 \;:\;   \pi _{\mathcal L} J=L \\ J\subset K } }    
\langle   x , h ^{w } _{J} \rangle_{w} ^2 
\le  (\rho-1) \sum_{\substack{L'\in \mathcal L \;:\; L'\subset K\\ L' \textup{ is maximal}}} 
\sum_{ J\in \mathcal Q_2 \;:\;  J \subset L'}   
\langle   x , h ^{w } _{J} \rangle_{w} ^2 \,. 
\end{equation}
Recall that $ \rho-1 = \frac 1 {16}$. 
We can estimate 
\begin{align*} 
 \sum_{\substack{ J\in \mathcal Q_2 \;:\;   \pi _{\mathcal L} J=L \\ J\subset K } }    
\langle   x , h ^{w } _{J} \rangle_{w} ^2  
 &
 \le   \frac 1 {16}
\sum_{\substack{ J\in \mathcal Q_2 \;:\;    J\subset L } }    
 \langle   x , h ^{w } _{J} \rangle_{w} ^2  
\\
&\le  \frac {\tau ^2} {16} \cdot   \frac{\lvert  K\rvert ^2 \cdot \sigma (K)}{ P(\sigma (L-K), K) ^2}  \,. 
\end{align*}
The last inequality follows from the definition of size,  and finishes the proof of \eqref{e:small1<}. 
\end{proof}

The collections below are the first contribution to $  \mathcal Q ^{\textup{large}}$.  
Take $  \mathcal Q ^{\textup{large}}_1 := \cup \{  \mathcal Q ^{\textup{large}} _{L,1} \;:\; L\in \mathcal L\}$, where 
\begin{equation*}
  \mathcal Q ^{\textup{large}} _{L,1} := \{ Q \in  \mathcal Q _ { L , 1}    \;:\; \tilde Q_1 = L\} \,. 
\end{equation*}
Note that Lemma~\ref{l:holes} applies to this Lemma, take the collection $ \mathcal S$ of that Lemma to be  the singleton $ \{L\}$. 
From the mutual orthogonality \eqref{e:subadd},  we then have 
\begin{equation*}
\mathbf B _{\mathcal Q ^{\textup{large}} _{1} } \le \sqrt 2 \sup _{L\in \mathcal L} \mathbf B _{\mathcal Q ^{\textup{large}} _{L,1} }  
\lesssim \tau \,. 
\end{equation*}

The collections $ \mathcal Q _{L,t}$, for $ L\in \mathcal L$, and $ t\ge 2$ are the second contribution to $  \mathcal Q ^{\textup{large}}$, namely 
\begin{equation*}
 \mathcal Q ^{\textup{large}} _{2} := \bigcup _{L
 \in \mathcal L} \bigcup _{t \ge 2}  \mathcal Q _ { L,t} \,. 
\end{equation*}
For them, we need to estimate $ \mathbf B _{\mathcal Q _{L,t}}$.  
\begin{lemma}\label{l:Y} There holds 
\begin{equation}\label{e:Y}
 \mathbf B _{\mathcal Q _{L,t}} \lesssim \rho ^{-t/2} \tau \,. 
\end{equation}
\end{lemma}

From this, we can conclude from \eqref{e:subadd} that 
\begin{align*}
\mathbf B _{ \mathcal Q ^{\textup{large}} _2} 
& \le \sum_{t \ge 2 } \mathbf B _{\bigcup  \{\mathcal Q   _{L,t} \;:\; L\in \mathcal L \} } 
\\
& \le  \sqrt 2 \sum_{t \ge2 } \sup _{L\in \mathcal L}\mathbf B _{\mathcal Q _{L,t} } 
 \lesssim \tau \sum_{t\ge 2 } \rho  ^{-t/2} \lesssim \tau \,. 
\end{align*}

\begin{proof} 
For $ L \in \mathcal L$, let  $ \mathcal S_{L}$ be  the $ \mathcal L$-children of $ L$.  
For each $ Q\in  \mathcal Q_{L,t} $, we must have $ Q_2 \subset \pi _{\mathcal S_L} Q_2 \subset \tilde Q_1 $. 
Then, divide the collection $ \mathcal Q _{L,t}$ into  three collections  $ \mathcal Q ^ \ell _{L,t} $, $ \ell =1,2,3$,  where  
\begin{align*}
 \mathcal Q ^{1} _{L,t} &:= \{Q \in \mathcal Q _{L,t} \;:\; Q_2 \Subset \pi _{\mathcal S_L} Q_2\} , \qquad t \in \mathbb N 
\\
\mathcal Q ^{2} _{L,t} &:= \{Q \in \mathcal Q _{L,t} \;:\; Q_2 \not\Subset \pi _{\mathcal S_L} Q_2 \Subset \tilde Q_1\} ,\qquad  1\le t \leq r+1 , 
\\
  \mathcal Q ^{3} _{L,t} &:= \mathcal Q _{L,t} - ( \mathcal Q ^{1} _{L,t}\cup  \mathcal Q ^{2} _{L,t}  ),\qquad  1\le t \leq r+1 . 
\end{align*}
For $ t> r+1$, we necessarily have $Q_2 \Subset \pi _{\mathcal S_L} Q_2 $, hence $  \mathcal Q ^{1} _{L,t} =  \mathcal Q _{L,t} $. 
(The integer $ r$ is associated with goodness, and the definition of $ J\Subset I$.)  
\smallskip 

We treat them in turn. The collections $  \mathcal Q ^{1} _{L,t} $  fit the hypotheses of Lemma~\ref{l:holes},  just take the collection of intervals $ \mathcal S$ of that Lemma to be  $ \mathcal S_{L}$.
It follows that  $\mathbf B _{ \mathcal Q ^{1} _{L,t} } \lesssim \boldsymbol \beta (t)$, 
where the latter is the best constant in the inequality 
\begin{equation} \label{e:CS}   
\sum_{J\in (\mathcal Q_ {L,t})_2 \;:\; J\Subset K} 
P(\sigma (I_0- K), J) ^2 \bigl\langle \frac {x} {\lvert  J\rvert } , h ^{w} _{J} \bigr\rangle_{w} ^2 
\le \boldsymbol \beta (t) ^2   \sigma (K), \qquad  K \in \mathcal S _L,\  L\in \mathcal L,\ t \ge 2 \,.  
\end{equation}
  
\smallskip 

By \eqref{e:eta<size}, we have an estimate without decay in $ t$, $ \boldsymbol \beta (t) \lesssim \textup{size} (\mathcal Q)$. 
Use the estimate for $ t\le r+3$, say.   
In the case of $ t> r+3$, the  essential property  is \eqref{e:ddecay}.
The left hand side of \eqref{e:CS} is dominated by the sum below.  
Note that we index the sum first over $ L'$, which are  $ r+1$-fold  $ \mathcal L$-children of $ K$, whence $ L'\Subset K$,  followed by $ t-r-2$-fold $ \mathcal L$-children of 
$ L'$.  
\begin{align}
\sum_{\substack{L'\in \mathcal L \\   \pi _{\mathcal L} ^{r+1} L'=K }}  &
\sum_{\substack{L''\in \mathcal L \\   \pi ^{t-r-2} _{\mathcal L} L''= L'}}  
\sum_{J \in \mathcal Q_2 \;:\; J\subset L''}
P(\sigma (I_0-K), J) ^2 \bigl \langle\frac x {\lvert  J\rvert }  , h ^{w} _{J} \bigr\rangle _{w} ^2 
\\
&\stackrel {{\eqref{e:PP}}}\le 
\sum_{\substack{L'\in \mathcal L \\   \pi _{\mathcal L} ^{r+1} L'=K }}  
\frac {P(\sigma (I_0-K), L') ^2} {\lvert  L'\rvert ^2  }
\sum_{\substack{L''\in \mathcal L \\   \pi ^{t-r-2} _{\mathcal L} L''= L'}}   
\sum_{J \in \mathcal Q_2 \;:\; J\subset L''} \langle x  , h ^{w} _{J} \rangle _{w} ^2  
 \\  \label{e:inOut}
 &\stackrel {\eqref{e:ddecay}} \lesssim  \rho ^{-t+r+2} 
\sum_{\substack{L'\in \mathcal L \\   \pi _{\mathcal L} ^{r+1} L'=K }}  
\frac {P(\sigma (I_0-K), L') ^2} {\lvert  L'\rvert ^2  } 
\sum_{J \in \mathcal Q_2 \;:\; J\subset L'}
\langle x  , h ^{w} _{J} \rangle _{w} ^2   
\\
&\stackrel {\phantom{\eqref{e:ddecay}}} \lesssim \rho ^{-t} \tau ^2 
\sum_{\substack{L'\in \mathcal L \\   \pi _{\mathcal L} ^{r+1} L'=K }}   \sigma (L') 
\lesssim \tau ^2 \rho ^{-t}  \sigma (K) \,. 
 \end{align}
We have also used \eqref{e:PP}, and then   the central property 
\eqref{e:ddecay} following from the construction of $ \mathcal L$,  finally appealing  to  the definition of size.   
Hence,  $  \boldsymbol \beta (t) \lesssim \tau ^2 \rho ^{-t}$. 
This completes the analysis of $  \mathcal Q ^{1} _{L,t} $. 

\bigskip 
We need only consider the collections $  \mathcal Q ^{2} _{L,t} $ for $ 1\le t \le r+1$, and they fall under the scope of Lemma~\ref{l:Holes}. 
And, we see immediately that we have $ \mathbf B _{  \mathcal Q ^{2} _{L,t}} \lesssim \tau $. 

Similarly, we  need only consider the collections $  \mathcal Q ^{3} _{L,t} $ for $ 1\le t \le r+1$.   
In this case we have $ Q_2 \not\Subset \pi _{\mathcal S_L}Q_2 \not\Subset \tilde Q_1$.  
It follows that we must have $ 2 ^{r} \le \lvert  Q_1\rvert/\lvert  Q_2\rvert \le 2 ^{2r+2}  $.
Namely, this ratio can take only one of a 
finite number of values, implying that Lemma~\ref{l:equal} applies easily to this case to complete the proof.  
\end{proof}

\subsubsection*{Pairs not strictly comparable to $ \mathcal L$}
It remains to consider the pairs $ Q\in \mathcal Q$ such that $ \tilde Q_1$ does not have a parent in $ \mathcal L$. 
The collection $  \mathcal Q ^{\textup{small}}_2 $  is taken to be the (much smaller) collection 
\begin{equation*}
\mathcal Q  ^{\textup{small}}_2  := \{Q \in \mathcal Q \;:\;   \textup{$ {Q_2}$ does not have a parent in $ \mathcal L$}\}\,.  
\end{equation*}
Observe that 
$
\textup{size} (\mathcal Q  ^{\textup{small}}_2 ) \le \sqrt { (\rho -1)}  \tau \le \frac \tau 4 
$.
This is as required for this collection.\footnote{The collections $ \mathcal Q  ^{\textup{small}}_1$ and $ \mathcal Q  ^{\textup{small}}_2$ are also mutually orthogonal, but this fact is not needed for our proof.}

\begin{proof}
Suppose  $ \eta <\textup{size} (\mathcal Q ^{\textup{small}}_2 ) $. Then, there is an interval $ K \in \widetilde {(\mathcal Q ^{\textup{small}}_1 )}_1 \cup (\mathcal Q  ^{\textup{small}}_2)_2 $ so that 
\begin{align*}
\eta ^2   \sigma (K) \le 
\frac {P (\sigma (I_0-K), K) ^2 } {\lvert  K\rvert ^2  }
\sum_{\substack{ J \in (\mathcal Q ^{\textup{small}}_2 )_2 \\ J \subset K}} 
\langle  x    , h ^{w} _{J} \rangle _{w} ^2\,. 
\end{align*}
Suppose that $ K$ does not contain any interval in $ \mathcal L$. 
It follows from the initial intervals added to $ \mathcal L$, see \eqref{e:Kdef}, that we must have $ \eta \le \frac \tau 4$.

Thus, $ K$ contains an interval in $ \mathcal L$. This means that $ K$ must fail the inequality \eqref{e:Lconstruct}.  
Therefore,   we have 
\begin{align*}
\eta ^2 \sigma (K) & \le (\rho -1)
\frac {P (\sigma (I_0-K), K) ^2 } {\lvert  K\rvert ^2  }
\sum_{\substack{ J \in \mathcal Q_2 \\ J \subset K}} 
\langle  x    , h ^{w} _{J} \rangle _{w} ^2
 \le \frac {\tau ^2 } {16} \sigma (K) \,. 
\end{align*}
This relies upon the definition of size, and proves our claim.
\end{proof}

For the pairs not yet in one of  our collections, it must be that $ Q_2$ has a parent in $ \mathcal L$, but not $ \tilde Q_1$. 
Using $ \mathcal L ^{\ast} $, the maximal intervals in $ \mathcal L$, divide them into the three collections 
\begin{align}
 \mathcal Q ^{\textup{large}} _{3} 
& := \{ Q\in \mathcal Q \;:\;  
Q_2 \Subset  \pi _{\mathcal L ^{\ast} } Q_2 \subset \tilde Q_1 
 \} ,
 \\
 \mathcal Q ^{\textup{large}} _{4} 
& := \{ Q\in \mathcal Q \;:\;  
Q_2 \not\Subset  \pi _{\mathcal L ^{\ast} } Q_2 \Subset \tilde Q_1 
 \} , 
 \\
 \mathcal Q ^{\textup{large}} _{5} & 
 := \{ Q\in \mathcal Q \;:\;  
Q_2 \not\Subset  \pi _{\mathcal L ^{\ast} } Q_2 \subsetneq \tilde Q_1 
 , \textup{and}\  \pi _{\mathcal L ^{\ast} } Q_2 \not\Subset  \tilde Q_1  \} \,. 
\end{align}

Observe that Lemma~\ref{l:holes} applies  to give 
\begin{equation} \label{e:Q3}
\mathbf B _{ \mathcal Q ^{\textup{large}}_3} \lesssim \tau  \,. 
\end{equation}
Take the collection $ \mathcal S$ of Lemma~\ref{l:holes} to be $ \mathcal L ^{\ast} $, and use \eqref{e:eta<size}.

Observe that Lemma~\ref{l:Holes} applies to show that the  estimate \eqref{e:Q3} holds for $  \mathcal Q ^{\textup{large}}_4$. 
Take $ \mathcal S$ of that Lemma to be $ \mathcal L ^{\ast} $.  The estimate from Lemma~\ref{l:Holes} is given in terms of $ \eta $, as defined in \eqref{e:Holes}.  But,  it is at most $ \tau $.  

In the last collection, $  \mathcal Q ^{\textup{large}}_5$, notice that the conditions placed upon the pair implies that 
$ \lvert  Q_1\rvert\le 2 ^{2r+2} \lvert  Q_2\rvert  $, for all $ Q\in  \mathcal Q ^{\textup{large}}_5$.  It therefore follows from a straight forward application of Lemma~\ref{l:equal}, that \eqref{e:Q3} holds for this collection as well. 

\subsection{Upper Bounds on the  Stopping Form}\label{s:upper}

We have three lemmas that prove upper bounds on the norm of the stopping form in situations in which 
there is some decoupling between the martingale difference on $ g$, and the argument of the Hilbert transform.
First, an elementary observation. 
 
\begin{proposition}\label{p:}
For intervals $ J\subset L\Subset K$,  with $ L$ either good, or the child of a good interval, 
 \begin{equation} \label{e:PP}
\frac {P(\sigma (I_0-K), J) } {\lvert  J\rvert } \simeq   \frac {P(\sigma (I_0-K), L) } {\lvert  L\rvert } \,. 
\end{equation}
\end{proposition}

\begin{proof}
The property of interval $ I$ being good, Part I \cite{12014319}, says that if $ I \subset \tilde I$, and $ 2 ^{r-1} \lvert  I\rvert \le \lvert  \tilde I\rvert  $, 
then the distance of either child of $ I$ to the boundary of $ \tilde I$ is at least $ \lvert  I\rvert ^{\epsilon } \lvert  \tilde I\rvert ^{1- \epsilon } $.   
Thus, in the case that $ L$ is the child of a good interval,  the parent $ \hat L  $ of $ L$ is contained in $ K$, and $ 2 ^{r-1} \lvert  \hat  L\rvert\le \lvert  K\rvert  $, so by the definition of goodness, 
 \begin{align*}
\textup{dist} (J, I_0 - K)  & \ge \textup{dist} (L, I_0 - K)  
\\
& \ge   \lvert    L\rvert ^{\epsilon } \lvert  K\rvert ^{1- \epsilon } 
\ge 2 ^{ r(1- \epsilon) } \lvert  L\rvert\,.   
\end{align*}
The same inequality holds if $ L $ is good.  
Then, one has the equivalence above, by inspection of the Poisson integrals.   
\end{proof}

\begin{lemma}\label{l:holes}  Let $ \mathcal S$ be a collection of pairwise disjoint intervals in $ I_0$. 
Let $ \mathcal Q$ be admissible such that for each $ Q\in \mathcal Q$, there is an $ S\in \mathcal S $ 
with $ Q_2 \Subset S \subset \tilde Q_1$.  
 Then, there holds 
\begin{gather}
\lvert  B _{\mathcal Q} (f,g)\rvert \lesssim  \eta  \lVert f\rVert_{\sigma } \lVert g\rVert_{w} , 
\\
\label{e:S<} 
\textup{where} \quad 
\eta ^2 := 
\sup _{S\in \mathcal S} \frac {1} {\sigma (S)    }
\sum_{J \in \mathcal Q_2 \;:\; J\Subset S} 
P(\sigma (I_0- S), J) ^2 \bigl\langle \frac x {\lvert  J\rvert } , h ^{w} _{J} \rangle_{w} ^2 \,. 
\end{gather}
\end{lemma}
 
It is useful to note that $ \eta $ is always smaller than the size: 
For $ S\in \mathcal S$,   The condition $  Q_2 \Subset S$ implies that $ S$ cannot be contained in an energy stopping interval 
$ \mathcal F _{\textup{energy}} (I_0)$.  Let $ \mathcal J ^{\ast} $ be the maximal intervals $ J\in \mathcal Q_2$ with $ J\Subset S$, 
and note that  goodness, via \eqref{e:PP},  applies to see that 
\begin{align}
\sum_{J \in \mathcal Q_2 \;:\; J\Subset S} 
P(\sigma (I_0- S), J) ^2 \bigl\langle \frac {x} {\lvert  J\rvert } , h ^{w} _{J} \bigr\rangle_{w} ^2 
& = 
\sum_{ J ^{\ast}  \in \mathcal J ^{\ast} } 
\sum_{J \in \mathcal Q_2 \;:\; J\subset J ^{\ast} } 
P(\sigma (I_0- S), J) ^2 \bigl\langle \frac {x} {\lvert  J\rvert } , h ^{w} _{J} \bigr\rangle_{w} ^2 
\\
& \lesssim 
\sum_{ J ^{\ast}  \in \mathcal J ^{\ast} }  \frac {P(\sigma (I_0- S), J ^{\ast} ) ^2 } {\lvert  J ^{\ast} \rvert ^2  }
\sum_{J \in \mathcal Q_2 \;:\; J\subset J ^{\ast} }  \langle x, h ^{w} _{J} \rangle_{w} ^2 
\\  \label{e:eta<size}   
& \lesssim   \sum_{ J ^{\ast}  \in \mathcal J ^{\ast} } \sigma (J ^{\ast} ) \lesssim 
 \textup{size} (\mathcal Q) ^2  \sigma (S).  
\end{align}
Above, we have argued as follows.  If $ J ^{\ast} \in \mathcal J ^{\ast} $ is not contained in any  interval 
$ S'\in\mathcal F _{\textup{energy}} (I_0)$, then 
\begin{equation*}
\frac {P(\sigma (I_0- S), J ^{\ast} ) ^2 } {\lvert  J ^{\ast} \rvert ^2  }
\sum_{J \in \mathcal Q_2 \;:\; J\subset J ^{\ast} }  \langle x, h ^{w} _{J} \rangle_{w} ^2\le \textup{size} (\mathcal Q) ^2  \sigma (J ^{\ast} ) 
\end{equation*}
by the definition of size.  If $ J ^{\ast} $ is however contained in an interval $ S'\in \mathcal F _{\textup{energy}} (I_0)$, 
it follows that

\begin{proof}
An  interesting part of the proof is that it depends very much on cancellative properties of the martingale differences of $ f$. 
(Absolute values must be taken \emph{outside} the sum defining the stopping form!)

Assume that the Haar support of $ f$ is contained in $ \mathcal Q_1$.  
Take $ \mathcal F$ and $ \alpha _{f} ( \cdot )$ to be stopping data 
defined in this way:  First, add to $ \mathcal F$ the interval $ I_0$, and set 
$ \alpha _{f} (I_0):= \mathbb E ^{\sigma } _{I_0} \lvert  f\rvert $. 
Inductively, if $ F\in \mathcal F$ is minimal, add to $ \mathcal F$ the maximal children $ F'$ such that  $ \alpha _{f} (F') :=   \mathbb E _{F'}  ^{\sigma } \lvert f \rvert > 4 \alpha _{f} (F)$.  
We have $ \sum_{F\in \mathcal F} \alpha _{f} (F) ^2 \sigma (F) \lesssim \lVert f\rVert_{\sigma } ^2 $.  And, so there holds 
\begin{equation}\label{e:quasi}
\sum_{F\in \mathcal F} \alpha _{f} (F) \sigma (F) ^{1/2}  \lVert Q ^{w} _{F} g\rVert_{w} 
\lesssim \lVert f\rVert_{\sigma } \lVert g\rVert_{w} , 
\end{equation}
for a family of mutually orthogonal projections $ Q ^{w} _{F}$ acting on $ L ^2 (w)$. 
Following \cite{12014319} we call this the \emph{quasi-orthogonality argument}.

Write the bilinear form as 
\begin{align} 
B _{\mathcal Q} (f,g) 
&= \sum_{ J  }\langle  H _{\sigma } \varphi _{J} ,  \Delta _{J} ^{w} g \rangle _{w} 
\\
\label{e:zvf} \textup{where} \quad 
\varphi _{J } &  := \sum_{ \substack{Q  \in \mathcal Q\;:\; Q_2=J }}  
 \mathbb E ^{\sigma } _{J} \Delta ^{\sigma }_ {Q_1} f \cdot (I_0 - \tilde Q_1)  \,. 
\end{align}
The function $ \varphi _J$ is well-behaved, as we now explain.  
At each point  $ x$ with  $  \varphi _{J} (x) \neq 0$, the sum above is over pairs $ Q$ such that 
 $ Q_2=J$ and $ x\in I_0-\tilde Q_1 $. 
 By the convexity property of admissible collections,   the sum is over consecutive (good)  martingale  differences of $ f$. 
The basic telescoping property of these differences shows that the sum is bounded by the stopping value $ \alpha _{f} (\pi_{\mathcal F}  J)$. 
Let $ I ^{\ast }$ be the maximal interval of the form $ \tilde Q_1$ with $ x\in I_0- \tilde Q_1$, and let $ I _{\ast} $
be the child of the minimal such interval which contains $ J$.  
Then, 
\begin{align} \label{e:Zvf}
\begin{split}
\lvert  \varphi _  {J} (x) \rvert& = \Bigl\lvert  \sum_{\substack{Q \in \mathcal Q \;:\;   Q_2=J \\ x \in I - \tilde Q_1} } 
\mathbb E ^{\sigma } _{J} \Delta ^{\sigma } _{Q_1} f  (x)
\Bigr\rvert
\\&= \bigl\lvert \mathbb E ^{\sigma } _{I^\ast } f  -  \mathbb E ^{\sigma } _{I_ \ast } f  \bigr\rvert 
\lesssim  \alpha _{f} (\pi_{\mathcal F}  J) (I_0 -S),  
\end{split}
\end{align}
 where $ S$ is the $ \mathcal S$-parent of $ J$. 

We can estimate as below, for $ F\in \mathcal F$:  
\begin{align} 
\Xi (F)&  := \Bigl\lvert   \sum_{Q\in \mathcal Q \;:\; \pi _{\mathcal F} Q_2 =F} 
\mathbb E _{Q_2} \Delta ^{\sigma } _{Q_1 }f \cdot \langle  H _{\sigma } ( I_0 - \tilde Q_1) ,  \Delta _{J} ^{w} g \rangle _{w} 
\Bigr\rvert
\\  & \stackrel{\eqref{e:zvf}}= 
 \Bigl\lvert\sum_{J \in \mathcal Q_2 \;:\;  \pi _{\mathcal F}J =F }    
 \langle H _{\sigma } \varphi _{J},  \Delta ^{w} _{J} g\rangle _{w} 
\Bigr\rvert  
\\&  
 \stackrel {\eqref{e:Zvf}}\lesssim  
\alpha _{f} (F) 
\sum_{\substack{S\in \mathcal S \\ \pi _{\mathcal F}S=F}} 
 \sum_{\substack{J\in \mathcal Q_2  \\  J \subset S }}  
 P(\sigma (I_0 - S), J ) \bigl\lvert \bigl\langle  \frac x {\lvert   J\rvert }
 ,  \Delta ^{w} _{J} g \bigr\rangle_w\bigr\rvert 
 \\
 &
 \stackrel { \phantom{\eqref{e:Zvf}}}\lesssim  \alpha _{f} (F)  \Bigl[
 \sum_{\substack{S\in \mathcal S \\ \pi _{\mathcal F}S=F}} 
 \sum_{\substack{J\in \mathcal Q_2  \\  J \subset S }}  
 P(\sigma (I_0 - S), J ) ^2  \bigl\langle \frac x {\lvert  J\rvert }, h ^{w} _{J} \bigr\rangle_{w} ^2  
  \times  \sum_{\substack{ J\in \mathcal Q_2\\  \pi _{\mathcal F}J =F }}   \hat g (J) ^2  
 \Bigr] ^{1/2} 
 \\
 & \stackrel {\eqref{e:S<}}\lesssim \eta   \alpha _{f} (F)  
  \Bigl[
   \sum_{\substack{S\in \mathcal S \\ \pi _{\mathcal F}S=F}}  \sigma (S) \times 
\sum_{\substack{J\in \mathcal Q_2 \\  \pi _{\mathcal F}J =F }}   \hat g (J) ^2  
 \Bigr] ^{1/2} 
 \\
 &  \stackrel {\phantom{\eqref{e:S<}}}\lesssim \eta  \alpha _{f} (F)  \sigma (F) ^{1/2} 
   \Bigl[
\sum_{\substack{J\in \mathcal Q_2 \;:\;   \pi _{\mathcal F}J =F }}   \hat g (J) ^2  
 \Bigr] ^{1/2}  \,. 
\end{align}
The top line follows from \eqref{e:zvf}. 
In the second, we appeal to \eqref{e:Zvf} and  monotonicity  principle, see \cite{12014319}*{\S4}, the latter being available to us since $ J \subset S$ implies $ J\Subset S$, by hypothesis.   
We also take advantage of the strong assumptions on the intervals in $ \mathcal Q_2$: If $ J\in \mathcal Q_2$, we must have $ \pi _{\mathcal F}J = \pi _{\mathcal F} (\pi _{\mathcal S} J)$.  
The third line is Cauchy--Schwarz, followed by the appeal to the hypothesis \eqref{e:S<}, while the last line uses the fact that the intervals in $\mathcal S $ are pairwise disjoint.  

The  quasi-orthogonality argument \eqref{e:quasi} completes the proof, namely we have 
\begin{equation} \label{e:Xi}
\sum_{F\in \mathcal F} \Xi (F) \lesssim \textup{size}(\mathcal Q) \lVert f\rVert_{\sigma } \lVert g\rVert_{w} \,. 
\end{equation}
\end{proof}

\begin{lemma}\label{l:Holes}  Let $ \mathcal S$ be a collection of pairwise disjoint intervals in $ I_0$. 
Let $ \mathcal Q$ be admissible such that for each $ Q\in \mathcal Q$, there is an $ S\in \mathcal S $ 
with $ Q_2 \subset S \Subset \tilde Q_1$.  
 Then, there holds 
\begin{gather}
\lvert  B _{\mathcal Q} (f,g)\rvert \lesssim \eta  \lVert f\rVert_{\sigma } \lVert g\rVert_{w} , 
\\ \label{e:Holes}
\textup{where} \quad \eta ^2  := \sup _{S\in \mathcal S} 
 \frac {P(\sigma  (I_0-\pi _{ \tilde {\mathcal Q}_1} S), S) ^2 } {\sigma (S) \lvert  S\rvert ^2 } 
\sum_{J \in \mathcal Q _2 \;:\; J\subset S}  \langle  x , h ^{w} _{J} \rangle_{w} ^2  \,. 
\end{gather}
\end{lemma}

\begin{proof}
Construct stopping data $ \mathcal F$ and $ \alpha _{f} ( \cdot )$ as in the  proof of Lemma~\ref{l:holes}.   
The fundamental inequality \eqref{e:Zvf} is again used. 
Then,  by  the monotonicity principle, there holds for $ F\in \mathcal F$, 
\begin{align*}
\Xi (F) := &
\Bigl\lvert 
\sum_{Q\in \mathcal Q \;:\; \pi _{\mathcal F}\tilde Q_1=F} \mathbb E _{Q_2} \Delta ^{\sigma } _{Q_1} f \cdot 
\langle  H _{\sigma } (I_0- \tilde Q_1),  \Delta ^{w} _{Q_2} g\rangle_w 
\Bigr\rvert
\\
& \lesssim  \alpha _{f} (F)
\sum_{S\in \mathcal S \;:\; \pi _{\mathcal F}S=F} 
P(\sigma  (I_0-\pi _{ \tilde {\mathcal Q}_1} S),S ) \sum_{J\in \mathcal Q_2 \;:\;   J \subset S}  
 \bigl\langle \frac x {\lvert  S\rvert }, h ^{w} _{J} \bigr\rangle_{w} \cdot \lvert   \hat g (J)\rvert  
\\
& \lesssim 
\alpha _{f} (F) 
\Bigl[
\sum_{S\in \mathcal S \;:\; \pi _{\mathcal F}S=F} 
P(\sigma  (I_0-\pi _{ \tilde {\mathcal Q}_1} S),S ) ^2 
 \sum_{J\in \mathcal Q_2 \;:\;   J \subset S}  
 \bigl\langle \frac x {\lvert  S\rvert }, h ^{w} _{J} \bigr\rangle_{w} ^2 
\times 
\sum_{\substack{J\in \mathcal Q_2  \;:\;  \pi _{\mathcal F}J =F }} \hat g (J)  ^2 
\Bigr] ^{1/2} 
\\
& \lesssim \eta 
\alpha _{f} (F) 
\Bigl[
\sum_{S\in \mathcal S \;:\; \pi _{\mathcal F}S=F}  \sigma (S)
\times 
\sum_{\substack{J\in \mathcal Q_2  \;:\;  \pi _{\mathcal F}J =F }}  \hat g (J)  ^2 
\Bigr] ^{1/2} 
\\
& \lesssim 
\eta  \alpha _{f} (F) \sigma (F) ^{1/2} 
\Bigl[
\sum_{\substack{J\in \mathcal Q_2  \;:\;  \pi _{\mathcal F}J =F }}   \hat g (J) ^2  
\Bigr] ^{1/2}  \,. 
\end{align*}
After the monotonicity principle, we have used Cauchy--Schwarz, and the definition of $ \eta $.  
The quasi-orthogonality argument \eqref{e:quasi} then completes the analysis of this term, see \eqref{e:Xi}.
\end{proof}

The last Lemma that we need  is elementary, and is contained in the methods of \cite{10031596}. 

\begin{lemma}\label{l:equal} Let $ u\ge r$ be an integer, and  $ \mathcal Q$ be an admissible collection of pairs such that $\lvert  Q_1\rvert= 2 ^{u} \lvert  Q_2\rvert  $ for all $ Q\in \mathcal Q$.  There holds 
\begin{equation*}
\lvert  B _{\mathcal Q} (f,g) \rvert \lesssim \textup{size} (\mathcal Q)  \lVert f\rVert_{\sigma } \lVert g\rVert_{w}  \,. 
\end{equation*}
\end{lemma}

\begin{proof}
  Recall the form of the stopping form in \eqref{e:stop}.  
It is an elementary property of the Haar functions, that  
\begin{equation*}
\lvert  \mathbb E ^{\sigma } _{I_J} \Delta ^{\sigma }_I f\rvert \le \frac {\lvert  \hat f (I) \rvert  } {\sigma (I_J) ^{1/2} } \,. 
\end{equation*}
In addition,  from the monotonicity principle and the goodness of $ J$,  $\langle H _{\sigma } I_0 - I_J, h ^{w } _{J}\rangle_w \lesssim 
 P(\sigma (I_0-I_J), J )  \langle \frac x {\lvert  J\rvert }, h ^{w} _{J} \rangle_{w} $. 
Then, we have, keeping in mind that $ I_J$ is one or the other of the two children of $ I$, 
\begin{align*}
\lvert  B _{\mathcal Q} (f,g)\rvert 
& \le \sum_{I \in \mathcal Q_1}  {\lvert  \hat f (I) \rvert  } 
\sum_{\substack{J \;:\;  (I,J)\in \mathcal Q   }} {\sigma (I_J) ^{-1/2} }
  P(\sigma (I_0-I_J), J )  \bigl\langle \frac x {\lvert  J\rvert }, h ^{w} _{J} \bigr\rangle_{w} 
\lvert  \hat g (J) \rvert 
\\
& \le 
\lVert f\rVert_{\sigma } 
\biggl[
 \sum_{I \in \mathcal Q_1} 
 \biggl[ 
\sum_{\substack{J \;:\;  (I,J)\in \mathcal Q   }}  \frac {1} {\sigma (I_J)  } 
 P(\sigma (I_0-I_J), J )  \bigl\langle \frac x {\lvert  J\rvert }, h ^{w} _{J} \bigr\rangle_{w} 
\lvert  \hat g (J) \rvert 
 \biggr] ^2 
\biggr] ^{1/2} 
\\
& \le \textup{size} (\mathcal Q) \lVert f\rVert_{\sigma } \lVert g\rVert_{w} 
\end{align*}
This follows immediately from Cauchy--Schwarz, and the fact that for each $ J\in \mathcal Q _2$, there is a unique $ I\in \mathcal Q_1$ 
such that the pair $ (I,J)$ contribute  to the sum above. 
\end{proof}

\end{document}